\documentclass[12pt]{amsart}
\usepackage{amsfonts,amssymb,amscd,amsmath,enumerate,color,bm,multirow,bigdelim}
\usepackage[latin1]{inputenc}
\usepackage{amscd}
\usepackage{latexsym}
\usepackage{mathptmx}
\def\NZQ{\mathbb}               

\def\ZZ{{\NZQ Z}}
\def\RR{{\NZQ R}}

\def\frk{\mathfrak}               

\def\Phi{{\frk N}}
\def\ab{{\bold a}}

\def\eb{{\bold e}}

\def\xb{{\bold x}}
\def\yb{{\bold y}}

\def\vb{{\bold v}}

\def\opn#1#2{\def#1{\operatorname{#2}}} 
\opn\chara{char} 
\opn\length{\ell} 
\opn\pd{pd} 
\opn\rk{rk}
\opn\projdim{proj\,dim} 
\opn\injdim{inj\,dim} 
\opn\rank{rank}
\opn\depth{depth} 
\opn\grade{grade} 
\opn\height{height}
\opn\embdim{emb\,dim} 
\opn\codim{codim}

\opn\Tr{Tr} 
\opn\bigrank{big\,rank}
\opn\superheight{superheight}
\opn\lcm{lcm}
\opn\trdeg{tr\,deg}
\opn\reg{reg} 
\opn\lreg{lreg} 
\opn\ini{in} 
\opn\lpd{lpd}
\opn\size{size}
\opn\mult{mult}
\opn\dist{dist}
\opn\cone{cone}
\opn\lex{lex}
\opn\rev{rev}
\opn\div{div} \opn\Div{Div} \opn\cl{cl} \opn\Cl{Cl}
\opn\Spec{Spec} \opn\Supp{Supp} \opn\supp{supp} \opn\Sing{Sing}
\opn\Ass{Ass} \opn\Min{Min}
\opn\Ann{Ann} \opn\Rad{Rad} \opn\Soc{Soc}
\opn\Syz{Syz} \opn\Im{Im} \opn\Ker{Ker} \opn\Coker{Coker}
\opn\Am{Am} \opn\Hom{Hom} \opn\Tor{Tor} \opn\Ext{Ext}
\opn\End{End} \opn\Aut{Aut} \opn\id{id} \opn\ini{in}

\opn\nat{nat}
\opn\pff{pf}
\opn\Pf{Pf} \opn\GL{GL} \opn\SL{SL} \opn\mod{mod} \opn\ord{ord}
\opn\Gin{Gin}
\opn\Hilb{Hilb}\opn\adeg{adeg}\opn\std{std}\opn\ip{infpt}
\opn\Pol{Pol}
\opn\sat{sat}
\opn\Var{Var}
\opn\Gen{Gen}

\opn\aff{aff} \opn\con{conv} \opn\relint{relint} \opn\st{st}
\opn\lk{lk} \opn\cn{cn} \opn\core{core} \opn\vol{vol}
\opn\link{link} \opn\star{star}
\opn\gr{gr}

\def\Hc{{\mathcal H}}

\def\Fc{{\mathcal F}}
\def\Pc{{\mathcal P}}

\def\Qc{{\mathcal Q}}

\def\pot#1#2{#1[\kern-0.28ex[#2]\kern-0.28ex]}

\opn\dirlim{\underrightarrow{\lim}}
\opn\inivlim{\underleftarrow{\lim}}
\def\Implies{\ifmmode\Longrightarrow \else
        \unskip${}\Longrightarrow{}$\ignorespaces\fi}
\def\implies{\ifmmode\Rightarrow \else
        \unskip${}\Rightarrow{}$\ignorespaces\fi}
\def\iff{\ifmmode\Longleftrightarrow \else
        \unskip${}\Longleftrightarrow{}$\ignorespaces\fi}

\let\:=\colon
\newtheorem{Theorem}{Theorem}[section]
\newtheorem{Lemma}[Theorem]{Lemma}
\newtheorem{Corollary}[Theorem]{Corollary}
\newtheorem{Proposition}[Theorem]{Proposition}

\newtheorem{Example}[Theorem]{Example}

\newtheorem{Question}[Theorem]{Question}
\newtheorem*{acknowledgement}{Acknowledgment}
\let\epsilon\varepsilon
\let\phi=\varphi
\let\kappa=\varkappa
\def\qed{\ifhmode\textqed\fi
	\ifmmode\ifinner\quad\qedsymbol\else\dispqed\fi\fi}
\def\textqed{\unskip\nobreak\penalty50
	\hskip2em\hbox{}\nobreak\hfil\qedsymbol
	\parfillskip=0pt \finalhyphendemerits=0}
\def\dispqed{\rlap{\qquad\qedsymbol}}

\opn\dis{dis}
\opn\height{height}
\opn\dist{dist}
\def\pnt{{\raise0.5mm\hbox{\large\bf.}}}

\opn\Lex{Lex}

\newcommand{\vbig}[1]{\multicolumn{2}{c}{$\mbox{\smash{\Huge $#1$}}$}}
\begin{document}
\title{Reflexive polytopes arising from edge polytopes}
\author[T. Nagaoka]{Takahiro Nagaoka}
\address[Takahiro Nagaoka]{Department of Mathematics,
	Graduate School of Science,
	Kyoto University,
	Kyoto, 606-8522, Japan}
\email{tnagaoka@math.kyoto-u.ac.jp}
\author[A. Tsuchiya]{Akiyoshi Tsuchiya}
\address[Akiyoshi Tsuchiya]{Department of Pure and Applied Mathematics,
	Graduate School of Information Science and Technology,
	Osaka University,
	Suita, Osaka 565-0871, Japan}
\email{a-tsuchiya@ist.osaka-u.ac.jp}

\date{}
\begin{abstract}
	It is known that every lattice polytope is unimodularly equivalent to a face of some reflexive polytope.
	A stronger question is to ask whether every $(0,1)$-polytope is unimodularly equivalent to a facet of some
	reflexive polytope.
	A large family of $(0,1)$-polytopes are the edge polytopes of finite simple graphs.
	In the present paper, 
	it is shown that, by giving
	a new class of reflexive polytopes,
	each edge polytope is unimodularly equivalent to a facet of some reflexive polytope.
	Furthermore, we extend the characterization of normal edge polytopes to a characterization of normality for these new reflexive polytopes.
\end{abstract}
\maketitle

\section*{Introduction}
The reflexive polytope is one of the keywords belonging to the current trends 
in the research of convex polytopes.  In fact, many authors have studied 
reflexive polytopes from the viewpoints of combinatorics, commutative algebra
and algebraic geometry.
Hence to find new classes of reflexive polytopes is an important problem.

A \textit{lattice polytope} is a convex polytope all of whose vertices have integer coordinates.
A lattice polytope $\Pc \subset \RR^d$ of dimension $d$ is called \textit{reflexive} if the origin of $\RR^d$ is a unique lattice point belonging to the interior of $\Pc$ and its dual polytope
$$\Pc^\vee:=\{\yb \in \RR^d \ | \  \langle \xb,\yb \rangle \leq 1 \ \text{for all}\  \xb \in \Pc \}$$
is also a lattice polytope, where $\langle \xb,\yb \rangle$ is the usual inner product of $\RR^d$.
A reflexive polytope is often called a \textit{Gorenstein Fano} polytope.
It is known that reflexive polytopes correspond to Gorenstein toric Fano varieties, and they are related to
mirror symmetry (see, e.g., \cite{mirror,Cox}).
In each dimension, there exist only finitely many reflexive polytopes 
up to unimodular equivalence (\cite{Lag}) 
and all of them are known up to dimension $4$ (\cite{Kre}).
Moreover, 
it is known that every lattice polytope is unimodularly equivalent to a face of some reflexive polytope (\cite{refdim}).
Especially, we are interested in which lattice polytopes are unimodularly equivalent to a facet of some reflexive polytope. 

Now, we ask the following question:
\begin{Question}
	\label{q1}
Is every $(0,1)$-polytope unimodularly equivalent to a facet of some reflexive polytope?
\end{Question}

For a lattice polytope $\Pc$, its \textit{reflexive dimension} 
is the smallest integer $r$ such that $\Pc$ is unimodularly equivalent to a face of a reflexive polytope of dimension $r$.
Computing the reflexive dimension of a lattice polytope is hard problem in general.
However, if Question $\ref{q1}$ is true, it is reasonable to determine the reflexive dimension of a $(0,1)$-polytope.
In fact, Question $\ref{q1}$ is equivalent to the following:
\begin{Question}
	\label{q2}
	For any $(0,1)$-polytope of dimension $d$, 
	is its reflexive dimension equal to $d+1$? 
\end{Question}

Recently, this question has been shown for the following classes of $(0,1)$-polytopes (\cite{HT1,HT2,HT3}): 
\begin{itemize}
	\item  The order polytopes of partially ordered sets.
	\item  The chain polytopes of partially ordered sets. 
	\item  The stable set polytopes of perfect graphs.
\end{itemize}
To solve Question \ref{q2} for the above polytopes, the algebraic technique on Gr\"onbner bases is used.
In fact, these polytopes are normal and, in particular, compressed (see Section 3 for definitions).
However, a $(0,1)$-polytope may not be normal.
Hence we cannot solve this question for any $(0,1)$-polytope by using the same methods.
In the present paper, by using matrix theory, we show this question is true for the edge polytopes of finite simple graphs, which are not necessarily normal polytopes.

The paper is organized as follows.
In Section 1, we recall the definition of the edge polytopes of finite simple graphs.
In Section 2, we will give a new class of reflexive polytopes arising from some class of lattice polytopes (Theorem \ref{main}).
From this result, we can show that every edge polytope is unimodularly equivalent to a facet of some reflexive polytope (Corollary \ref{Cor:ref}).
Finally, in Section 3, we will give a criterion to determine the normality of the reflexive polytopes arising from the edge polytope of connected finite simple graph described in Corollary \ref{Cor:ref} (Theorem \ref{thm:normal}).

\begin{acknowledgement}{\rm
		The authors would like to thank anonymous referees for reading the manuscript carefully.
		The second author is partially supported by Grant-in-Aid for JSPS Fellows 16J01549.
	}
\end{acknowledgement}

\section{edge polytopes}
In this section, we introduce the edge polytopes of finite simple graphs.
First, we recall that two lattice polytopes $\Pc \subseteq \RR^d$ and $\Pc' \subseteq \RR^{d'}$ are said to be \emph{unimodularly equivalent} if there exists an affine map from the affine span $\aff(\Pc)$ of $\Pc$ to the affine span $\aff(\Pc')$ of $\Pc'$ that maps $\ZZ^d \cap \aff(\Pc)$ bijectively onto $\ZZ^{d'} \cap \aff(\Pc')$ and maps $\Pc$ to $\Pc'$. 

Now, we define the edge polytopes of finite simple graphs.
Recall that a finite \textit{simple} graph is a finite graph with no loops, no multiple edges and no isolated vertices.
Let $G$ be a simple graph on the vertex set $V(G)=[d]=\{1,\ldots,d\}$ and denote $E(G)$ the edge set of $G$.
The \textit{edge polytope} $\Pc_G \subset \RR^d$ of $G$ is the convex hull of all vectors $\eb_i+\eb_j$ such that
$\{i,j\} \in E(G)$, where $\eb_1,\ldots,\eb_d$ are the canonical unit vectors of $\RR^d$.
This means that the edge polytope of $\Pc_G$ of $G$ is the convex hull of all row vectors of the \textit{incidence matrix} $A_G$ of $G$, where $A_G$ is the matrix in $\{ 0,1 \}^{E(G) \times [d]}$ with 
\begin{align*}
a_{e,v}=
\begin{cases}
1&\text{if\ } v \in e,\\
0&\text{otherwise}.
\end{cases}
\end{align*}
Moreover, the dimension of $P_G$ equals $\text{rank}(A_G)-1$.
In fact, 
\begin{Lemma}[{\cite[p. 57]{VV}}]
	\label{dim}
	Let $G$ be a finite simple graph on $[d]$
	and $c_0(G)$ the number of connected bipartite components of $G$.
	Then the dimension of the edge polytope $\Pc_G$ of $G$ equals $d-c_0(G)-1$.
\end{Lemma}
Sometimes it is convenient to work with full-dimensional lattice polytopes, i.e., lattice polytopes embedded in a space of their same dimension.
However, the edge polytopes of finite simple graphs are not full-dimensional from Lemma \ref{dim}.
Given an edge polytope $\Pc_G$, one can easily get a full-dimensional unimodularly equivalent copy of $\Pc_G$ by considering the lattice polytope defined as the convex hull of the row vectors of $A_G$ with some columns deleted.
Indeed,
let $G_1,\ldots,G_k$ be the connected bipartite components of $G$.
If $k=0$, namely, $G$ is not bipartite, we consider the unimodular matrix
\[U=
\begin{pmatrix}
1& 0 & \cdots & 0 & 0  \\
1 & 1 & \ddots & \ddots &\vdots  \\
\vdots & 0 & \ddots & \ddots & \vdots  \\
\vdots & \vdots & \ddots & \ddots & 0 \\
1 & 0 & \cdots &0 & 1 \\
\end{pmatrix} \in \ZZ^{d \times d},
\]
where a matrix $A \in \ZZ^{d\times d}$ is \empty{unimodular} if $\det(A)=\pm 1$.
Then one has 
 $f_U({\mathcal{P}_G})=\{2\} \times \Qc_G$, where $f_U$ is the linear transformation in $\RR^d$ defined by $U$, i.e., $f_U(\bold{v})=\bold{v}U$ for all $\bold{v} \in \RR^d$, and $\Qc_G$ is the full-dimensional lattice polytope defined as the convex hull of the row vectors of $A_G$ with the $1$st column deleted. Hence $\Qc_G$ is
 a full-dimensional unimodularly equivalent copy of $\Pc_G$.
Next, assume that $k \geq 1$ and $V_{i1} \sqcup V_{i2}$ the bipartition of each $G_i$.
Then we may suppose that $\{i\} \in V_{i1}$ for each $i$ and $\{k+1\} \in V_{12}$.
Set $V_1=\bigsqcup_{1\leq i \leq k} V_{i1}$ and $V_2=\bigsqcup_{1\leq i \leq k} V_{i2}$.
Let $U=(u_{ij})_{1 \leq i ,j, \leq d} \in \ZZ^{d \times d}$ be the unimodular matrix such that 
\[u_{ij}=
\begin{cases}
1 & \text{if} \  i=j,\\
1 & \text{if}  \ 1 \leq j \leq k \ \text{and} \  i \in V_{j1},   \\
1 & \text{if}  \ 1 \leq j \leq k \ \text{and} \  i \in V_{2} \setminus V_{j2},\\
0 & \text{otherwise},
\end{cases}
\]
and 
$V=(v_{ij})_{1 \leq i ,j, \leq d} \in \ZZ^{d \times d}$ the unimodular matrix such that 
\[v_{ij}=
\begin{cases}
1 & \text{if} \  i=j,\\
1 & \text{if}  \ j=k+1 \ \text{and} \  i \in V_{2},\\
0 & \text{otherwise}.
\end{cases}
\]
Then since $V_{i1}$ and $V_{i2}$ are the bipartition of each $G_i$, one has 
$f_{VU}({\mathcal{P}_G})=\{1\}^{k+1} \times \Qc_G$, where
$\Qc_G$ is the full-dimensional lattice polytope defined as the convex hull of the row vectors of $A_G$ with the $1,\ldots, (k+1)$st columns deleted. Hence $\Qc_G$ is
a full-dimensional unimodularly equivalent copy of $\Pc_G$.
An example of this can be observed in Example \ref{full}.
\begin{Example}
	\label{full}
	{\em
	Let $G$ be the following finite simple graph with the incidence matrix of $A_G$.
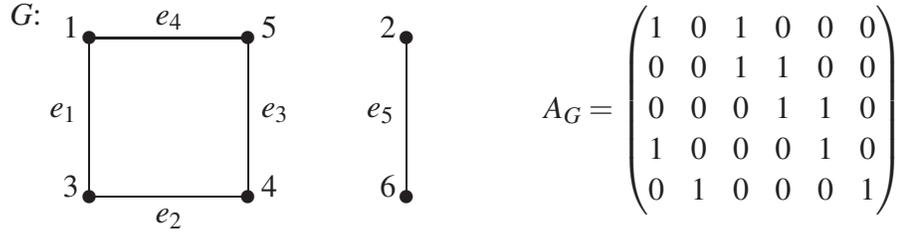
	\begin{center}
		\begin{tabular}{c}
	\begin{picture}(140,90)(90,90)
	\put(60,165){$G$:}
	\put(80,160){$1$}
	\put(80,100){$3$}
	\put(155,160){$5$}
	\put(155,100){$4$}
	\put(200,160){$2$}
	\put(200,100){$6$}
	\put(75,130){$e_1$}
	\put(155,130){$e_3$}
	\put(195,130){$e_5$}
	\put(115,90){$e_2$}
	\put(115,165){$e_4$}
	\put(90,160){\circle*{5}}
	\put(90,100){\circle*{5}}
	\put(150,160){\circle*{5}}
	\put(150,100){\circle*{5}}
	\put(210,160){\circle*{5}}
	\put(210,100){\circle*{5}}
	\put(90,160){\line(0,-1){60}}
	\put(150,160){\line(0,-1){60}}
	\put(210,160){\line(0,-1){60}}
	\put(90,160){\line(1,0){60}}
	\put(90,100){\line(1,0){60}}
	\end{picture}
\end{tabular}
 \qquad 
	$A_G=
	\begin{pmatrix}
	1& 0 & 1 & 0 & 0 &0 \\
	0 & 0 & 1 &1 & 0 & 0 \\
	0 & 0 & 0 & 1& 1 & 0 \\
	1 & 0 & 0 & 0& 1 & 0 \\
	0 & 1 & 0 & 0& 0 & 1 \\
	\end{pmatrix}$
\end{center}
	Then by Lemma \ref{dim}, $\Pc_G \subset \RR^6$ is a lattice polytope of dimension $3$.
	Set
	\[
		U=
	\begin{pmatrix}
	1& 0 & 0 & 0 & 0 &0 \\
	0 & 1 & 0 & 0& 0 & 0 \\
	0 & 1 & 1 & 0& 0 & 0 \\
	1 & 0 & 0 & 1& 0 & 0 \\
	0 & 1 & 0 & 0& 1 & 0 \\
	1 & 0 & 0 & 0& 0 & 1 \\
	\end{pmatrix}\ 
	\text{and} \ 	V=
	\begin{pmatrix}
	1& 0 & 0 & 0 & 0 &0 \\
	0 & 1 & 0 & 0& 0 & 0 \\
	0 & 0 & 1 & 0& 0 & 0 \\
	0 & 0 & 0 & 1& 0 & 0 \\
	0 & 0 & 1 & 0& 1 & 0 \\
	0 & 0 & 1 & 0& 0 & 1 \\
	\end{pmatrix}.
	 \]
	Then $U$ and $V$ are unimodular and one has 
	$f_{VU}({\mathcal{P}_G})=\{1\}^{3} \times \Qc_G$, where
	$\Qc_G$ is the full-dimensional lattice polytope defined as the convex hull of the row vectors of $A_G$ with the first, second and third columns deleted. Namely, $\Qc_G$ is the convex hull of the row vectors of 
		$$
	\begin{pmatrix}
	  0 & 0 & 0 \\
	  1 & 0& 0  \\
	  1 & 1& 0  \\
	  0 & 1& 0  \\
	  0 & 0& 1  \\
	\end{pmatrix}.$$
	Hence $\Qc_G$ is a full-dimensional unimodularly equivalent copy of $\Pc_G$.
}
\end{Example}

\section{Reflexive polytopes arising from edge polytopes}
In this section, we construct reflexive polytopes  which arise from the edge polytopes of finite simple graphs.
For two lattice polytopes $\Pc$  and $\Qc$ of dimension $d$ in $ \RR^n$, we set
$$\Omega(\Pc,\Qc)=\textnormal{conv}\{(\Pc\times\{1\}) \cup (-\Qc\times\{-1\})\}\subset \RR^{n+1},$$
where $-\Qc=\{-\ab \ | \ \ab \in \Qc \}$.
If $\Pc = \Qc$, then we will write $\Omega(\Pc)=\Omega(\Pc,\Pc)$.
We remark that the origin of $\RR^{n+1}$ is always a relative interior lattice point of $\Omega(\Pc)$. 
If $\Pc$ and $\Qc$ are full-dimensional, namely, $n=d$, then $\Omega(\Pc,\Qc)$ is a lattice polytope of dimension $d+1$, and each of $\Pc$ and $\Qc$ is a facet of $\Omega(\Pc,\Qc)$.

We show the following theorem:
\begin{Theorem}
	\label{main}
	Let $\Pc, \Qc \subset \RR^d$ be $(0,1,2)$-polytopes of dimension $d$ such that all of their vertices belong to 
	$$\{\mathbf{0}\} \cup \{\eb_i \mid 1 \leq i \leq d \} \cup\{\eb_i+\eb_j \mid 1 \leq i \leq j \leq d \} \subset \RR^d.$$
	If the origin of $\RR^{d+1}$ belongs to the interior of $\Omega(\Pc, \Qc)$, then $\Omega(\Pc, \Qc)$ is reflexive.
	In particular, $\Omega(\Pc)$ and $\Omega(\Qc)$ are reflexive.
\end{Theorem}

For two $d \times d$ integer matrices $A, B$, we write $A \thicksim B$ if $B$ can be obtained from $A$ by some row and column operations over $\ZZ$. 
In order to prove Theorem \ref{main}, we will need the following lemma and proposition.

\begin{Lemma}[{\cite[Corollary 35.6]{HibiRedBook}}]
	\label{facet}
	Let $\Pc \subset \RR^d$ be a lattice polytope of
	dimension $d$ containing the origin in its interior. 
	Then a point $\ab \in \RR^d$ is a vertex of $\Pc^\vee$ if 
	and only if $\Hc \cap \Pc$ is a facet of $\Pc$,
	where $\Hc$ is the hyperplane
	$$\left\{ \xb \in \RR^d \ | \ \langle \ab, \xb \rangle =1 \right\}$$
	in $\RR^d$.
\end{Lemma}

\begin{Proposition}
	\label{mat}
	Let $A=(a_{ij})_{1 \leq i,j \leq d} \in \{0,1,2\}^{d \times d}$ be a $d \times d$ integer matrix such that each row vector $\ab_i=(a_{i1},\ldots,a_{id})$ of $A$ satisfies the  following conditions:
	\begin{itemize}
		\item $a_{id} = 1$;
		\item $|a_{i1}+\cdots+a_{id-1}| \leq 2$.
	\end{itemize}
	If $\textnormal{det}(A) \neq 0$,
	then, for some integer $0 \leq s \leq d$, $$A \thicksim
	\begin{array}{rcccccccccll}
	\ldelim({7}{4pt}[]  &1   &       &       &       &     && &   \rdelim){7}{4pt}[] & \rdelim\}{3}{10pt}[$s$]\\
	&     &\ddots&       &       &\multicolumn{2}{c}{$\mbox{\smash{\Huge $0$}}$}     &  && \\
	&     &       &1       &       &        && &\\
	&     &       &       &2       &        &&  &\\
	&\multicolumn{2}{c}{$\mbox{\smash{\Huge $0$}}$} & &       &       &\ddots &&&  \\
	&     &       &        &       &          &  &2 &\\
	\end{array}.
	$$
	In this case, one has $2A^{-1} \in \ZZ^{d \times d}$.
\end{Proposition}

Since the proof of Proposition \ref{mat} is rather technical,
we save it until the end of this section.
We now prove Theorem \ref{main} and discuss a quick corollary.

\begin{proof}[Proof of Theorem \ref{main}]
Let $\Fc$ be a facet of $\Omega(\Pc, \Qc)$.
Then there exist $d+1$ vertices $\vb_1,\ldots,\vb_{d+1}$ of $\Omega(\Pc, \Qc)$ such that
$\textnormal{aff}(\vb_1,\ldots,\vb_{d+1}) \cap \Omega(\Pc, \Qc)=\Fc$,
where for a subset $A \subset \RR^{d+1}$, $\text{aff}(A)$ is the affine space generated by $A$.
Since the origin of $\RR^{d+1}$ belongs to the interior of $\Omega(\Pc, \Qc)$,
there exist rational numbers $a_1,\ldots,a_{d+1}$ such that 
$$\textnormal{aff}(\vb_1,\ldots,\vb_{d+1})=\{(x_1,\ldots,x_{d+1}) \in \RR^{d+1} \mid a_1x_1+\cdots+a_{d+1}x_{d+1} =1	 \}.$$
By Lemma \ref{facet}, we need to show that $a_1,\ldots,a_{d+1} \in \ZZ$.
Let $V$ be the $(d+1) \times (d+1)$ integer matrix whose $i$th row vector is $\vb_i$.
Then we have $\det(V) \neq 0$ and 
	\begin{displaymath}
V
\begin{pmatrix}
a_1\\
\vdots\\
a_{d+1}
\end{pmatrix}
=
\begin{pmatrix}
1\\
\vdots\\
1
\end{pmatrix}.
\end{displaymath}
Hence each  $a_i$ is the sum of all entries in the $i$th row  vector of $V^{-1}$.
If $\vb_{1},\ldots,\vb_{d+1}$ are vertices of $\Pc \times \{1\}$,
then one has $a_1=\cdots=a_d=0$ and $a_{d+1}=1$.
Similarly, if $\vb_{1},\ldots,\vb_{d+1}$ are vertices of $-\Qc \times \{-1\}$,
then one has $a_1=\cdots=a_d=0$ and $a_{d+1}=-1$.

Now, we assume that for some positive integer $1 \leq k \leq d$, 
$\vb_1,\ldots, \vb_k$ are vertices of $\Pc \times \{1\}$ and 
$\vb_{k+1},\ldots, \vb_{d+1}$ are vertices of $-\Qc \times \{-1\}$. 
Let $W$ be the $(d+1)\times (d+1)$ integer matrix 
such that for $1 \leq i \leq k$, the $i$th row vector is $\vb_i$
and for $k+1 \leq j \leq d+1$, the $j$th row vector is $-\vb_j$.
Then one has $\det(W) \neq 0$ and 
	\begin{displaymath}
W
\begin{pmatrix}
0\\
\vdots\\
0\\
1
\end{pmatrix}
=
\begin{pmatrix}
1\\
\vdots\\
1\\
1
\end{pmatrix}.
\end{displaymath}
Hence for any $1 \leq i \leq d+1$, the sum of all entries in the $i$th row vector of $W^{-1}$ is an integer.
Moreover, for $k+1 \leq i \leq d+1$, the $i$th column vector of $V^{-1}$  coincides with the $i$th column vector of $W^{-1}$ times $(-1)$.
Thus, since $2W^{-1}$ is an integer matrix from Proposition \ref{mat},
we know that for any $1 \leq i \leq d+1$, the sum of all entries in the $i$th row vector of $V^{-1}$ is an integer.
Indeed, let $(w_{i1},\ldots,w_{id+1})$ be the $i$th row vector of $W^{-1}$
and set $w=w_{i1}+\cdots+w_{id+1} \in \ZZ$.
Then since $a_i=w_{i1}+\cdots+w_{ik}-(w_{ik+1}+\cdots+w_{id+1})$ and $2W^{-1}$ is an integer matrix,
one has $w-a_i=2(w_{ik+1}+\cdots+w_{id+1}) \in \ZZ$.
Hence we obtain $a_i \in \ZZ$.
Therefore, by Lemma \ref{facet},  $\Omega(\Pc, \Qc)$ is reflexive, as desired.
\end{proof}

For a lattice polytope $\Pc \subset \RR^n$ of dimension $d$,
we say that $\Pc$ is reflexive if a full-dimensional unimodularly equivalent copy $\Qc \subset  \RR^d$ of $\Pc$ is reflexive. 
By Theorem \ref{main},
we can give a new class of reflexive polytopes arising from the edge polytopes of finite simple graphs
and we can determine the reflexive dimensions of the edge polytopes of finite simple graphs.
\begin{Corollary}
	\label{Cor:ref}
	Let $G$ be a finite simple graph on $[n]$.
	Then $\Omega(\Pc_{G}, \Pc_{G})$ is reflexive.
	Moreover, $\Pc_G$ is unimodularly equivalent to a facet of 
	$\Omega(\Pc_{G}, \Pc_{G})$.
\end{Corollary}

\begin{proof}
	Let $d$ be the dimension of $\Pc_G$ and $\Qc_G \subset \RR^d$ a full-dimensional unimodularly equivalent copy of $\Pc_G$ as defined in Section 1 (see Example \ref{full}).
	Then all of its vertices of $\Qc_G$ belong to 
		$$\{\mathbf{0}\} \cup \{\eb_i \mid 1 \leq i \leq d \} \cup\{\eb_i+\eb_j \mid 1 \leq i \leq j \leq d \} \subset \RR^d.$$
	Hence, it follows from Theorem \ref{main} that
	$\Omega(\Qc_G)$ is reflexive.
	Moreover  it is easy to see that $\Omega(\Qc_G)$ is a full-dimensional unimodularly equivalent copy of $\Omega(\Pc_G)$.
	Therefore, we know that  $\Omega(\Pc_G)$ is reflexive and 
	$\Pc_G$ is unimodularly equivalent to a facet of $\Omega(\Pc_G)$.
\end{proof}

We now end the section with a proof of Proposition \ref{mat}.
\begin{proof}[Proof of Proposition \ref{mat}]
	We prove this proposition by induction on $d$, i.e., the size of $A$.
	When $d=1$, the claim is trivial. 
	Suppose that $d >1$.
	We should show that for a given $d \times d$ integer matrix $A$ satisfying the assumption of the proposition, one can obtain a matrix from $A$ by some row and column operations over $\ZZ$ as the following:
	$$A \thicksim 
	\left(
	\begin{array}{c|ccc}
	1&0&\cdots&0\\
	\cline{1-4}
	0&&&\\
	\vdots&&\vbig{A'}\\
	0&&&\\
	\end{array}
	\right),$$\\
	where $A'$ is a $(d-1) \times (d-1)$ integer matrix satisfying the assumption of the proposition.
	
	We will divide the proof into some cases. Note that by the assumption, 
	no two row vectors in $A$ are the same,
	and each row vector $\ab_i$ of $A$ is one of the following:
	\begin{itemize}
		\item[(Type 1)] $\ab_i = (0,\ldots,0,1,0, \ldots,0, 1,0, \ldots,0, 1) $, i.e., for some $1 \leq j_1 < j_2 \leq d-1, a_{ij_1}=a_{ij_2}=a_{id}=1$, and for other $j,  a_{ij}=0$;
		\item[(Type 2)] $\ab_i = (0,\ldots,0, 1,0, \ldots,0, 1)$, i.e., for some $1 \leq j_1 \leq d-1, a_{ij_1}=a_{id}=1$, and for other $j,  a_{ij}=0$;
		\item[(Type 3)] $\ab_i = (0,\ldots,0,  2,0, \ldots,0, 1) $, i.e., for some $1 \leq j_1 \leq d-1, a_{ij_1}=2, a_{id}=1$, and for other $j, a_{ij}=0$;
		\item[(Type 4)] $\ab_i = (0,\dots,0,1) $, i.e., for any $1 \leq j \leq d-1, a_{ij}=0$, and $a_{id}=1$. 
	\end{itemize}
	We can then divide the proof into the following cases:
	\begin{itemize}
		\item[(1)] $A$ does not have any row vectors of Type 1 and Type 2;
		\item[(2)] $A$ has at least one row vector of Type 2, but $A$ has no row vectors of Type 1;
		\item[(3)] $A$ has at least one row vector of Type 1.
	\end{itemize}
	\smallskip
	The case (1):
	Since each row vector of A is either Type 3 or Type 4 and since $\textnormal{det}(A) \neq 0$, one can obtain the following matrices subsequently from $A$ by some row and column operations over $\ZZ$:
	$$A \thicksim
	\begin{pmatrix}
	2           &           &        &        &           &1 \\
	& \ddots &     &    \multicolumn{2}{c}{$\mbox{\smash{\Huge $0$}}$} &\vdots            \\
	&          & \ddots      &        &           &\vdots   \\
	\multicolumn{2}{c}{$\mbox{\smash{\Huge $0$}}$}  & & & 2 & 1&    \\
	&    &         &        &                &1   \\
	\end{pmatrix}
	\thicksim
	\begin{pmatrix}
	1           &           &        &        &           & \\
	& 2 &     &   & \multicolumn{2}{c}{$\mbox{\smash{\Huge $0$}}$} &            \\
	&          & \ddots      &        &           &   \\
	\multicolumn{2}{c}{$\mbox{\smash{\Huge $0$}}$}  & & & \ddots & &    \\
	&    &         &        &                &2   \\
	\end{pmatrix}.
	$$
	Thus, in this case, we can get the desired matrix from $A$ by some row and column operations over $\ZZ$. 
	\smallskip\\
	The case (2):
	Since $A$ has at least one row vector of Type 2 and $A$ does not have any row vector of Type 1, one can obtain the following matrix $\tilde{A}$ from $A$ by interchanging some rows and columns:
	$$\tilde{A}=
	\begin{array}{rcccccccccccll}
	\ldelim({7}{4pt}[] &1     &0     &\cdots&0        &	&         &           &             &                            &   & 1       &\rdelim){7}{4pt}[] & \rdelim\}{4}{10pt}[$r$]\\
	&0        &1     &\cdots&0        &		&         &           &             &                            &         &1&\\
	&\vdots &      &\ddots&\vdots &	&\multicolumn{4}{c}{$\mbox{\smash{\Huge $0$}}$}  &   &\vdots &\\
	&0        &0     &\cdots&1        &		&	     &           &             &                            &   &1        &\\ 
	\cline{2-12}
	&          &      &         &          &       &          &          &             &                            &   &        &\\ 
	&          &\multicolumn{8}{c}{$\mbox{\smash{\Huge $C$}}$}  &                                              &        &\\ 
	&          &      &         &          &        &         &          &             &                            &   &        &\\
	\end{array},
	$$
	where $r \geq 1$ and $C$ is a $(d-r) \times d$ integer matrix  such that each row vector is either Type 3 or Type 4.
	Now, by interchanging some row vectors of $C$ if necessarily, we can assume the first column vector of $C$ is  
	$(0,\ldots,0)^{\top}$
	or 
	$(2,0,\ldots,0)^{\top}$, where for a row vector $\ab$, $\ab^{\top}$ is the transpose of $\ab$. 
	If the first column vector of $C$ is $(0,\ldots,0)^{\top}$, we can obtain the below matrix from $\tilde{A}$ by a column operation: 
	$$\tilde{A} 
	= 
	\left(
	\begin{array}{ccccccccc}
	&&&&&&&&1\\
	&&&&&&&&1\\
	&\vbig{I_r}&&&\vbig{0}&&\vdots\\
	&&&&&&&&1\\
	\cline{1-9}
	0&&&&&&&&1\\
	0&&&&&&&&1\\
	\vdots&&&\vbig{C'}&&&&\vdots\\
	0&&&&&&&&1\\
	\end{array}
	\right)
	\thicksim
	\left(
	\begin{array}{ccccccccc}
	&&&&&&&&0\\
	&&&&&&&&1\\
	&\vbig{I_r}&&&\vbig{0}&&\vdots\\
	&&&&&&&&1\\
	\cline{1-9}
	0&&&&&&&&1\\
	0&&&&&&&&1\\
	\vdots&&&\vbig{C'}&&&&\vdots\\
	0&&&&&&&&1\\
	\end{array}
	\right)
	$$
	$$
	=:
	\left(
	\begin{array}{c|ccc}
	1&0&\cdots&0\\
	\cline{1-4}
	0&&&\\
	\vdots&&\vbig{A'}\\
	0&&&\\
	\end{array}
	\right),
	$$
	\smallskip\\
	where $I_r$ is the unit matrix of size $r$.
	Moreover, it is clear that $A'$ satisfies the assumption of the proposition.
	Hence, by the inductive hypothesis, one obtains the desired matrix from $A$.
	
	Assume that  the first column vector of $C$ is $(2,0,\ldots,0)^{\top}$.
	Then we can obtain the below matrices from $\tilde{A}$ subsequently by some row and column operations over $\ZZ$:
	$$\tilde{A}=
	\left(
	\begin{array}{ccccccccc}
	&&&&&&&&1\\
	&&&&&&&&1\\
	&\vbig{I_r}&&&\vbig{0}&&\vdots\\
	&&&&&&&&1\\
	\cline{1-9}
	2&0&&&\cdots&&&0&1\\
	\cline{1-9}
	0&&&&&&&&1\\
	\vdots&&&\vbig{C''}&&&&\vdots\\
	0&&&&&&&&1\\
	\end{array}
	\right)
	\thicksim
	\left(
	\begin{array}{ccccccccc}
	&&&&&&&&0\\
	&&&&&&&&1\\
	&\vbig{I_r}&&&\vbig{0}&&\vdots\\
	&&&&&&&&1\\
	\cline{1-9}
	0&0&&&\cdots&&&0&-1\\
	\cline{1-9}
	0&&&&&&&&1\\
	\vdots&&&\vbig{C''}&&&&\vdots\\
	0&&&&&&&&1\\
	\end{array}
	\right)
	$$
	$$
	\thicksim
	\left(
	\begin{array}{ccccccccc}
	&&&&&&&&0\\
	&&&&&&&&1\\
	&\vbig{I_r}&&&\vbig{0}&&\vdots\\
	&&&&&&&&1\\
	\cline{1-9}
	0&0&&&\cdots&&&0&1\\
	\cline{1-9}
	0&&&&&&&&1\\
	\vdots&&&\vbig{C''}&&&&\vdots\\
	0&&&&&&&&1\\
	\end{array}
	\right)
	=:
	\left(
	\begin{array}{c|ccc}
	1&0&\cdots&0\\
	\cline{1-4}
	0&&&\\
	\vdots&&\vbig{A''}\\
	0&&&\\
	\end{array}
	\right).
	$$\smallskip\\
	Moreover, it is clear that $A''$ satisfies the assumption of the proposition.
	Hence, by induction hypothesis, one obtain a desired matrix from $A$.
	\smallskip\\
	The case (3):
	Since $A$ has at least one row vector of Type 1, we can obtain the following matrix from $A$ by interchanging some rows and columns of $A$:
	$$\tilde{A}:=
	\left(
	\begin{array}{cccccc}
	1&1&0&\cdots&0&1\\
	\cline{1-6}
	&&&&&\\
	&&\vbig{B}&&\\
	\cline{1-6}
	&&&&&\\
	&&\vbig{C}&&\\
	\end{array}
	\right),
	$$
	where $B$ is a $m \times d$ integer matrix for some 
	$m \geq 0$ such that each row vector is either Type 1 or Type 2, and $C$ is a $(d-m-1) \times d$  integer matrix such that each row vector is either Type 3 or Type 4. We set $l:=d-m-1$. 
	Let  $B_{\{1,2\}}$ (resp. $C_{\{1,2\}}$) denote the submatrix consisting of the first and second column vectors of $B$ (resp. $C$).
	Here, to prove the claim, we divide into the following subcases:
	\begin{itemize}
		\item[(3-1)] $B_{\{1,2\}}$ is a zero matrix;
		\item[(3-2)] $B_{\{1,2\}}$ is not a zero matrix.
	\end{itemize}
	Note, in both subcases, by a permutation of first row and second row, we can assume $$C_{\{1,2\}} = \begin{pmatrix}
	c_0&c_1\\
	0&c_2\\
	\vdots&\vdots\\
	0&c_l\\
	\end{pmatrix},
	$$ where either $c_0=2$ and $c_1=0$, or $c_0=0$.
	\smallskip\\
	The subcase (3-1):	
	In this case, if $c_0=0$, then we can obtain the below matrix from $\tilde{A}$ by some column operations: 
	
	$$\tilde{A}=
	\left(
	\begin{array}{cc|ccccc|c}
	1&1&0&&\cdots&&0&1\\
	\cline{1-8}
	0&0&&&&&&1\\
	\vdots&\vdots&&&\vbig{B'}&&\vdots\\
	0&0&&&&&&1\\
	\cline{1-8}
	0&c_1&&&&&&1\\
	\vdots&\vdots&&&\vbig{C'}&&\vdots\\
	0&c_l&&&&&&1\\
	\end{array}
	\right)
	\thicksim
	\left(
	\begin{array}{cc|ccccc|c}
	1&0&0&&\cdots&&0&0\\
	\cline{1-8}
	0&0&&&&&&1\\
	\vdots&\vdots&&&\vbig{B'}&&\vdots\\
	0&0&&&&&&1\\
	\cline{1-8}
	0&c_1&&&&&&1\\
	\vdots&\vdots&&&\vbig{C'}&&\vdots\\
	0&c_l&&&&&&1\\
	\end{array}
	\right)
	$$
	$$
	=:
	\left(
	\begin{array}{c|ccc}
	1&0&\cdots&0\\
	\cline{1-4}
	0&&&\\
	\vdots&&\vbig{A'}\\
	0&&&\\
	\end{array}
	\right).
	$$\\
	Moreover, it is clear that $A'$ satisfies the assumption of the proposition. Hence, by the inductive hypothesis, we obtain the desired matrix from $A$.
	Next, if $c_0=2 \ \text{and} \ c_1=0$, then we can obtain the below matrices from $\tilde{A}$ subsequently by some row and column operations: 
	$$\tilde{A}=
	\left(
	\begin{array}{cc|ccccc|c}
	1&1&0&&\cdots&&0&1\\
	\cline{1-8}
	0&0&&&&&&1\\
	\vdots&\vdots&&&\vbig{B''}&&\vdots\\
	0&0&&&&&&1\\
	\cline{1-8}
	2&0&0&&\cdots&&0&1\\
	\cline{3-7}
	0&c_2&&&&&&1\\
	\vdots&\vdots&&&\vbig{C''}&&\vdots\\
	0&c_l&&&&&&1\\
	\end{array}
	\right)
	\thicksim
	\left(
	\begin{array}{cc|ccccc|c}
	1&1&0&&\cdots&&0&1\\
	\cline{1-8}
	0&0&&&&&&1\\
	\vdots&\vdots&&&\vbig{B''}&&\vdots\\
	0&0&&&&&&1\\
	\cline{1-8}
	0&-2&0&&\cdots&&0&-1\\
	\cline{3-7}
	0&c_2&&&&&&1\\
	\vdots&\vdots&&&\vbig{C''}&&\vdots\\
	0&c_l&&&&&&1\\
	\end{array}
	\right)
	$$
	$$
	\thicksim
		\left(
		\begin{array}{cc|ccccc|c}
		1&0&0&&\cdots&&0&0\\
		\cline{1-8}
		0&0&&&&&&1\\
		\vdots&\vdots&&&\vbig{B''}&&\vdots\\
		0&0&&&&&&1\\
		\cline{1-8}
		0&-2&0&&\cdots&&0&-1\\
		\cline{3-7}
		0&c_2&&&&&&1\\
		\vdots&\vdots&&&\vbig{C''}&&\vdots\\
		0&c_l&&&&&&1\\
		\end{array}
		\right)
	\thicksim
	\left(
	\begin{array}{cc|ccccc|c}
	1&0&0&&\cdots&&0&0\\
	\cline{1-8}
	0&0&&&&&&1\\
	\vdots&\vdots&&&\vbig{B''}&&\vdots\\
	0&0&&&&&&1\\
	\cline{1-8}
	0&2&0&&\cdots&&0&1\\
	\cline{3-7}
	0&c_2&&&&&&1\\
	\vdots&\vdots&&&\vbig{C''}&&\vdots\\
	0&c_l&&&&&&1\\
	\end{array}
	\right)$$
	$$
	=:
	\left(
	\begin{array}{c|ccc}
	1&0&\cdots&0\\
	\cline{1-4}
	0&&&\\
	\vdots&&\vbig{A''}\\
	0&&&\\
	\end{array}
	\right).
	$$
	Moreover, it is clear that $A''$ satisfies the assumption of the proposition. Hence, by the inductive hypothesis, we obtain the desired matrix from $A$.
	\smallskip\\
	The subcase (3-2):
	Note, since $\textnormal{det}(A) \neq 0$, each row vector of $B_{\{1,2\}}$ is $(0,1)$, $(1,0)$, or $(0,0)$.
	Thus, by some row permutations of $\tilde{A}$,
	 and if necessary a permutation of the first and second columns, we can assume 
	$$B_{\{1,2\}} = 
	\begin{array}{rccll}
	\ldelim({10}{4pt}[] &0&1&\rdelim){10}{4pt}[]&\rdelim\}{3}{10pt}[$p$] \\
	&\vdots&\vdots& & \\
	&0&1& & \\
	&1&0& & \rdelim\}{3}{10pt}[$q-p$,] \\
	&\vdots&\vdots& & \\
	&1&0& & \\
	&0 &0 & & \rdelim\}{3}{10pt}[$m-q$]\\
	&\vdots&\vdots& & \\
	&0&0&& \\
	\end{array}$$
	where $p \geq 1$ and $q-p \geq 0$.
	
	Let $\tilde{\ab}_2$ be the $2$nd row vector of $\tilde{A}$.
	Then $\tilde{\ab}_2$ is either Type 1 or Type 2.
	If $\tilde{\ab}_2$ is Type 2, we can obtain the below matrices from $\tilde{A}$ subsequently by some row operations over $\ZZ$ as the following: 
	
	$$\tilde{A}=
	\left(
	\begin{array}{cc|ccccc|c}
	1&1&0&&\cdots&&0&1\\
	\cline{1-8}
	0&1&0&&\cdots&&0&1\\
	\cline{3-7}
	0&1&&&&&&1\\
	\vdots&\vdots&&&&&&\vdots\\
	0&1&&&&&&1\\
	1&0&&&&&&1\\
	\vdots&\vdots&&&\vbig{B'}&&\vdots\\
	1&0&&&&&&1\\
	0&0&&&&&&1\\
	\vdots&\vdots&&&&&&\vdots \\
	0&0&&&&&&1\\
	\cline{1-8}
	c_0&c_1&&&&&&1\\
	0&c_2&&&&&&1\\
	\vdots&\vdots&&&\vbig{C'}&&\vdots\\
	0&c_l&&&&&&1\\
	\end{array}
	\right)
	\thicksim
	\left(
	\begin{array}{cc|ccccc|c}
	1&0&0&&\cdots&&0&0\\
	\cline{1-8}
	0&1&0&&\cdots&&0&1\\
	\cline{3-7}
	0&1&&&&&&1\\
	\vdots&\vdots&&&&&&\vdots\\
	0&1&&&&&&1\\
	1&0&&&&&&1\\
	\vdots&\vdots&&&\vbig{B'}&&\vdots\\
	1&0&&&&&&1\\
	0&0&&&&&&1\\
	\vdots&\vdots&&&&&&\vdots \\
	0&0&&&&&&1\\
	\cline{1-8}
	c_0&c_1&&&&&&1\\
	0&c_2&&&&&&1\\
	\vdots&\vdots&&&\vbig{C'}&&\vdots\\
	0&c_l&&&&&&1\\
	\end{array}
	\right)
	$$
	$$
	\thicksim
	\left(
	\begin{array}{cc|ccccc|c}
	1&0&0&&\cdots&&0&0\\
	\cline{1-8}
	0&1&0&&\cdots&&0&1\\
	\cline{3-7}
	0&1&&&&&&1\\
	\vdots&\vdots&&&&&&\vdots\\
	0&1&&&&&&1\\
	0&0&&&&&&1\\
	\vdots&\vdots&&&\vbig{B'}&&\vdots\\
	0&0&&&&&&1\\
	0&0&&&&&&1\\
	\vdots&\vdots&&&&&&\vdots \\
	0&0&&&&&&1\\
	\cline{1-8}
	0&c_1&&&&&&1\\
	0&c_2&&&&&&1\\
	\vdots&\vdots&&&\vbig{C'}&&\vdots\\
	0&c_l&&&&&&1\\
	\end{array}
	\right)
	=:
	\left(
	\begin{array}{c|ccc}
	1&0&\cdots&0\\
	\cline{1-4}
	0&&&\\
	\vdots&&\vbig{A'}\\
	0&&&\\
	\end{array}
	\right).
	$$\\
	Moreover, it is clear that $A'$ satisfies the assumption of the proposition. Hence, by the inductive hypothesis, we obtain the desired matrix from $A$.
	
	Next assume that $\tilde{\ab}_2$ is Type 1.
	Note that $\tilde{\ab}_2$ is a row vector like $(0,1,0, \ldots,0 ,1, 0,\ldots,0 ,1)$. By interchanging columns, we can assume $\tilde{\ab}_2=(0,1,1,0,\ldots,0 ,1)$. Then, we can obtain the below matrices from $\tilde{A}$ subsequently by some row and column operations
	over $\ZZ$ as follows: 
	
	$$\tilde{A}=
	\left(
	\begin{array}{ccc|ccccc|c}
	1&1&0&0&&\cdots&&0&1\\
	\cline{1-9}
	0&1&1&0&&\cdots&&0&1\\
	\cline{4-8}
	0&1&b_2&&&&&&1\\
	\vdots&\vdots&\vdots&&&&&&\vdots\\
	0&1&b_p&&&&&&1\\
	1&0&b_{p+1}&&&&&&1\\
	\vdots&\vdots&\vdots&&&\vbig{B''}&&\vdots\\
	1&0&b_q&&&&&&1\\
	0&0&b_{q+1}&&&&&&1\\
	\vdots&\vdots&\vdots&&&&&&\vdots \\
	0&0&b_m&&&&&&1\\
	\cline{1-9}
	c_0&c_1&c_1'&&&&&&1\\
	0&c_2&c_2'&&&&&&1\\
	\vdots&\vdots&\vdots&&&\vbig{C''}&&\vdots\\
	0&c_l&c_l'&&&&&&1\\
	\end{array}
	\right)
	\thicksim
	\left(
	\begin{array}{ccc|ccccc|c}
	1&0&-1&0&&\cdots&&0&0\\
	\cline{1-9}
	0&1&1&0&&\cdots&&0&1\\
	\cline{4-8}
	0&1&b_2&&&&&&1\\
	\vdots&\vdots&\vdots&&&&&&\vdots\\
	0&1&b_p&&&&&&1\\
	1&0&b_{p+1}&&&&&&1\\
	\vdots&\vdots&\vdots&&&\vbig{B''}&&\vdots\\
	1&0&b_q&&&&&&1\\
	0&0&b_{q+1}&&&&&&1\\
	\vdots&\vdots&\vdots&&&&&&\vdots \\
	0&0&b_m&&&&&&1\\
	\cline{1-9}
	c_0&c_1&c_1'&&&&&&1\\
	0&c_2&c_2'&&&&&&1\\
	\vdots&\vdots&\vdots&&&\vbig{C''}&&\vdots\\
	0&c_l&c_l'&&&&&&1\\
	\end{array}
	\right)
	$$
	$$
	\thicksim
	\left(
	\begin{array}{ccc|ccccc|c}
	1&0&0&0&&\cdots&&0&0\\
	\cline{1-9}
	0&1&1&0&&\cdots&&0&1\\
	\cline{4-8}
	0&1&b_2&&&&&&1\\
	\vdots&\vdots&\vdots&&&&&&\vdots\\
	0&1&b_p&&&&&&1\\
	1&0&b_{p+1}+1&&&&&&1\\
	\vdots&\vdots&\vdots&&&\vbig{B''}&&\vdots\\
	1&0&b_q+1&&&&&&1\\
	0&0&b_{q+1}&&&&&&1\\
	\vdots&\vdots&\vdots&&&&&&\vdots \\
	0&0&b_m&&&&&&1\\
	\cline{1-9}
	c_0&c_1&c_1'+c_0&&&&&&1\\
	0&c_2&c_2'&&&&&&1\\
	\vdots&\vdots&\vdots&&&\vbig{C''}&&\vdots\\
	0&c_l&c_l'&&&&&&1\\
	\end{array}
	\right)
	\thicksim
	\left(
	\begin{array}{ccc|ccccc|c}
	1&0&0&0&&\cdots&&0&0\\
	\cline{1-9}
	0&1&1&0&&\cdots&&0&1\\
	\cline{4-8}
	0&1&b_2&&&&&&1\\
	\vdots&\vdots&\vdots&&&&&&\vdots\\
	0&1&b_p&&&&&&1\\
	0&0&b_{p+1}+1&&&&&&1\\
	\vdots&\vdots&\vdots&&&\vbig{B''}&&\vdots\\
	0&0&b_q+1&&&&&&1\\
	0&0&b_{q+1}&&&&&&1\\
	\vdots&\vdots&\vdots&&&&&&\vdots \\
	0&0&b_m&&&&&&1\\
	\cline{1-9}
	0&c_1&c_1'+c_0&&&&&&1\\
	0&c_2&c_2'&&&&&&1\\
	\vdots&\vdots&\vdots&&&\vbig{C''}&&\vdots\\
	0&c_l&c_l'&&&&&&1\\
	\end{array}
	\right)$$
	
	$$
	=:
	\left(
	\begin{array}{c|ccc}
	1&0&\cdots&0\\
	\cline{1-4}
	0&&&\\
	\vdots&&\vbig{A''}\\
	0&&&\\
	\end{array}
	\right).
	$$\\
	One can easily show $A''$ also satisfies the assumption of the proposition. 
	Therefore, this completes the proof by induction.
\end{proof} 

\section{Normality}
Let $\Pc \subset \RR^n$ be a lattice polytope of dimension $d$
and $\Qc \subset \RR^d$  a full-dimensional unimodularly equivalent copy of $\Pc$.
Suppose that every lattice point in $\ZZ^d$ is an affine integer combination of the lattice points in $\Qc$.
We remark that a full-dimensional unimodularly equivalent copy of the edge polytope of any connected finite simple graph satisfies this condition (see the proof of \cite[Corollary 3.4]{HMT}).
We say that  $\Pc$ is \textit{normal} if for any positive integer $N$ and for any lattice point $\xb \in N\Qc \cap \ZZ^d$, there exist just $N$ lattice points $\xb_1,\ldots,\xb_N \in \Qc \cap \ZZ^d$ such that $\xb=\xb_1+\cdots+\xb_N$,
where $N\Qc=\{N\ab \mid \ab \in \Qc \}$.
Normal polytopes are popular objects in commutative algebra and toric algebraic geometry.
Moreover, this property is often called the \textit{integer decomposition property}, where the integer decomposition property is particularly important in the theory and application of integer programing \cite[\S 22.10]{integer}.

In this section, we discuss the normality of the reflexive polytopes arising from the edge polytopes of connected finite simple graphs described in Corollary \ref{Cor:ref}.
First, we introduce a criterion to determine the normality of the edge polytopes of connected finite simple graphs.
\begin{Theorem}[{\cite[Corollarly 2.3]{OH:normal}}]
	Let $G$ be a connected finite simple graph on $[n]$.
	Then $\Pc_G$ is normal if and only if for any two odd cycles $C$ and $C'$ of $G$ having no common vertex, there exists an edge of $G$ joining a vertex of $C$ with a vertex of $C'$. 
\end{Theorem}

The following theorem gives a criterion to determine the normality of the reflexive polytopes arising from the edge polytopes of connected finite simple graphs described in Corollary \ref{Cor:ref}.
\begin{Theorem}
	\label{thm:normal}
	Let $G$ be a connected finite simple graph on $[n]$.
	Then $\Omega(\Pc_G)$ is normal if and only if $G$ does not contain two disjoint odd cycles.
\end{Theorem}

In order to prove this theorem, we need the following lemma.
Recall that a lattice polytope is called \textit{unimodular} if all its triangulations are unimodular, that is each simplex has the normalized volume equal to $1$.
Moreover, a lattice polytope is called \textit{compressed} if all its pulling triangulations are unimodular (refer to \cite{Sul}).
In particular, unimodular lattice polytopes are compressed and normal.

\begin{Lemma}[{\cite[Example 3.6 b]{OHH:pure}}]
	\label{Lem:unimodular}
	Let $G$ be a connected finite simple graph on $[n]$.
	Then $\Pc_G$ is unimodular if and only if $G$ does not contain two disjoint odd cycles.
\end{Lemma}

Now, we prove Theorem \ref{thm:normal}.
\begin{proof}[Proof of Theorem \ref{thm:normal}]
First, let us assume that $G$ has two disjoint odd cycles $C_1$ and $C_2$.
Then it follows from Lemma \ref{dim} that the dimension of $\Pc_G$ equals $n-1$.
Moreover, we can assume that $V(C_1)=[2k+1]$ and $V(C_2)=\{2k+2,\ldots,2k+2\ell+2\}$ with some positive integers $k$ and $\ell$.
Let $\Qc_G \subset \RR^{n-1}$ be the full-dimensional unimodularly equivalent copy of $\Pc_G$ which is the convex hull of the row vectors of the incidence matrix $A_G$ of $G$ with $(2k+2\ell+2)$nd column deleted.
Then $\Omega(\Qc_G)$ is a full-dimensional unimodularly equivalent copy of $\Omega(\Pc_G)$ and 
one has $$\Omega(\Qc_G) \cap \ZZ^n=((\Qc_G\times\{1\})\cap \ZZ^n)\cup ((-\Qc_G\times\{-1\})\cap \ZZ^n) \cup \{(0,\ldots,0) \}.$$
We show that $\Omega(\Qc_G)$ is not normal.
Set 
$$\xb=\eb_1+\cdots+\eb_{2k+1}-(\eb_{2k+2}+\cdots+\eb_{2k+2\ell+1})+(k-\ell)\eb_n \in \ZZ^n.$$
Since $\eb_1+\eb_2,\ldots,\eb_{2k}+\eb_{2k+1}$ and  $\eb_1+\eb_{2k+1}$ are vertices of $\Qc_G$
and since $-\eb_{2k+2}-\eb_{2k+3},\ldots,-\eb_{2k+2\ell}-\eb_{2k+2\ell+1},-\eb_{2k+2\ell+1}$ and $-\eb_{2k+2}$ are vertices of  $-\Qc_G$,
it follows that $\xb \in (k+\ell+1)\Omega(\Qc_G) \cap \ZZ^n$.
Suppose that $\Omega(\Qc_G)$ is normal.
Then there exist just $k+m+1$ lattice points $\xb_1,\ldots,\xb_{k+m+1} \in \Omega(\Qc_G) \cap \ZZ^{n}$ such that $\xb=\xb_1+\cdots+\xb_{k+\ell+1}$.
For any vertex $\vb$ of  $\Qc_G \times \{1\}$, one has 
$\langle \vb, \eb_1+\cdots+\eb_{2k+1} \rangle \in \{0,1,2\}$ and $\langle \vb, \eb_{2k+2}+\cdots+\eb_{2k+2\ell+1} \rangle \in \{0,1,2\}$.
Hence since $\langle \xb, \eb_1+\cdots+\eb_{2k+1} \rangle =2k+1$ and $\langle \xb, \eb_{2k+2}+\cdots+\eb_{2k+2\ell+1} \rangle = 2\ell$,
we can assume that $\xb_1,\ldots,\xb_{k+1} \in \Qc_G \times \{1\}$ and 
$\xb_{k+2},\ldots,\xb_{k+\ell+1} \in -\Qc_G \times \{-1\}$.
Then one has $\langle \xb_1+\cdots+\xb_{k+\ell+1}, \eb_n\rangle=k-\ell+1$.
Thus, $\xb \neq \xb_1+\cdots+\xb_{k+\ell+1}$, a contradiction.
Therefore, $\Omega(\Qc_G)$ is not normal.

Conversely, assume that $G$ does not have two disjoint odd cycles.
Let $\Qc_G \subset \RR^d$ be a full-dimensional unimodularly equivalent copy of $\Pc_G$ defined in Section 1 (see Example \ref{full}).
Hence $\Omega(\Qc_G)$ is a full-dimensional unimodularly equivalent copy of $\Omega(\Pc_G)$ and 
one has $$\Omega(\Qc_G) \cap \ZZ^{d+1}=((\Qc_G\times\{1\})\cap \ZZ^{d+1})\cup ((-\Qc_G\times\{-1\})\cap \ZZ^{d+1}) \cup \{{(0,\ldots,0)} \}.$$
Let $\Delta$ be a pulling triangulation of $\Omega(\Qc_G)$ such that all its maximal simplices contain the origin of $\RR^{d+1}$. 
We will show that $\Delta$ is unimodular.
Let $\sigma$ be a maximal simplex of $\Delta$.
Then there exist $d+1$ lattice points $\xb_1,\ldots,\xb_{d+1}$ belonging to $\Qc_G \times \{1\}$ such that 
$$\sigma=\text{conv}\{{\mathbf 0},\vb_1,\ldots,\vb_t, -\vb_{t+1},\ldots,-\vb_{d+1}\}$$
with some integer $0 \leq t \leq d+1$.
It follows that the normalized volume of $\sigma$ equals $|\det (V)|$,
where $V$ is the $(d+1) \times (d+1)$ integer matrix whose $i$th row vector is $\vb_i$.
Set
$$\sigma'=\text{conv}\{{\mathbf 0},\vb_1,\ldots,\vb_{d+1}\}.$$
Then the normalized volume of $\sigma$ is equal to that of $\sigma'$.
We show that the normalized volume of $\sigma'$ is $1$.
In general, for a lattice polytope $\Pc \subset \RR^d$,
the lattice polytope $$\text{Pyr}(\Pc)=\text{conv}(\Pc \times \{0\}, \eb_{d+1}) \subset \RR^{d+1}$$ 
is called the \textit{lattice pyramid} over $\Pc$.
It is known that the normalized volume of $\Pc$ equals that of $\text{Pyr}(\Pc)$.
Let
$\tau$ be the lattice simplex which is the convex hull of the row vectors of $V$ with the last column deleted.
Then $\tau$ is a simplex of dimension $d$ all of whose vertices belonging to $\Qc_G$.
By Lemma \ref{Lem:unimodular} and \cite[Theorem 5.6.3]{dojo}, 
the normalized volume of any maximal simplex all of whose vertices belonging to $\Qc_G$ is $1$, that of $\tau$ is also $1$.
Since $\text{Pyr}(\tau)$ is unimodularly equivalent to $\sigma'$, the normalized volume of $\sigma'$ equals $1$.
Hence $\Delta$ is unimodular.
\end{proof}

\end{document}